\documentclass[12pt, reqno]{amsart}

\usepackage{latexsym}
\usepackage{amssymb}
\usepackage{mathrsfs}
\usepackage{amsmath}
\usepackage{fancybox,color}
\usepackage{enumerate}
\usepackage[latin1]{inputenc}
\usepackage{eurosym}

\newcommand{\kom}[1]{}
\renewcommand{\kom}[1]{{\bf [#1]}}

 \def\1{\raisebox{2pt}{\rm{$\chi$}}}

\newtheorem{theorem}{Theorem}[section]
\newtheorem{corollary}[theorem]{Corollary}
\newtheorem{lemma}[theorem]{Lemma}

\newtheorem{remark}[theorem]{Remark}

\newcommand{\R}{{\mathbb R}}

\newcommand{\N}{{\mathbb N}}
\newcommand{\Z}{{\mathbb Z}}

 \def\1{\raisebox{2pt}{\rm{$\chi$}}}
 
\def\vint_#1{\mathchoice%
          {\mathop{\kern 0.2em\vrule width 0.6em height 0.69678ex depth -0.58065ex
                  \kern -0.8em \intop}\nolimits_{\kern -0.4em#1}}%
          {\mathop{\kern 0.1em\vrule width 0.5em height 0.69678ex depth -0.60387ex
                  \kern -0.6em \intop}\nolimits_{#1}}%
          {\mathop{\kern 0.1em\vrule width 0.5em height 0.69678ex
              depth -0.60387ex
                  \kern -0.6em \intop}\nolimits_{#1}}%
          {\mathop{\kern 0.1em\vrule width 0.5em height 0.69678ex depth -0.60387ex
                  \kern -0.6em \intop}\nolimits_{#1}}}
\def\vintslides_#1{\mathchoice%
          {\mathop{\kern 0.1em\vrule width 0.5em height 0.697ex depth -0.581ex
                  \kern -0.6em \intop}\nolimits_{\kern -0.4em#1}}%
          {\mathop{\kern 0.1em\vrule width 0.3em height 0.697ex depth -0.604ex
                  \kern -0.4em \intop}\nolimits_{#1}}%
          {\mathop{\kern 0.1em\vrule width 0.3em height 0.697ex depth -0.604ex
                  \kern -0.4em \intop}\nolimits_{#1}}%
          {\mathop{\kern 0.1em\vrule width 0.3em height 0.697ex depth -0.604ex
                  \kern -0.4em \intop}\nolimits_{#1}}}

\newcommand{\aveint}[2]{\mathchoice%
          {\mathop{\kern 0.2em\vrule width 0.6em height 0.69678ex depth -0.58065ex
                  \kern -0.8em \intop}\nolimits_{\kern -0.45em#1}^{#2}}%
          {\mathop{\kern 0.1em\vrule width 0.5em height 0.69678ex depth -0.60387ex
                  \kern -0.6em \intop}\nolimits_{#1}^{#2}}%
          {\mathop{\kern 0.1em\vrule width 0.5em height 0.69678ex depth -0.60387ex
                  \kern -0.6em \intop}\nolimits_{#1}^{#2}}%
          {\mathop{\kern 0.1em\vrule width 0.5em height 0.69678ex depth -0.60387ex
                  \kern -0.6em \intop}\nolimits_{#1}^{#2}}}

\newcommand{\T}{\mathbb T}

\begin{document}

\address{Theresa C. Anderson, Department of Mathematics, Purdue University, 150 N. University St., W. Lafayette, IN
47907, USA.}
\email{tcanderson@purdue.edu}

\address{Jos\'e Madrid: Department of  Mathematics,  University  of  California,  Los  Angeles,  Portola Plaza 520, Los  Angeles, CA, 90095, USA.}
\email{jmadrid@math.ucla.edu}

\title{New bounds for discrete lacunary spherical averages}
\author{Theresa C. Anderson and Jos\'e Madrid}
\date{April 2020}
\begin{abstract}
    We show that the discrete lacunary spherical maximal function is bounded on $l^p(\mathbb{Z}^d)$ for all $p >\frac{d+1}{d-1}$.  Our range is new in dimension 4, where it appears that little was previously known for general lacunary radii.  Our technique follows that of Kesler-Lacey-Mena, using the Kloosterman refinement to improve the estimates in several places, which leads to an overall improvement in dimension 4. 
\end{abstract}

\maketitle
\section{Introduction and main results}
The full range of $l^p$ bounds for lacunary spherical maximal functions are still unknown in the discrete setting.  In this article we use refined number theoretic analysis which gives new bounds in dimension $4$ with no additional restrictions on the lacunary sequence.

We begin with a few definitions of our objects of study.
The discrete spherical averages are defined as
\[
A_{\lambda} f(n) = \frac{1}{N(\lambda)}\sum_{|m|^2 = \lambda}f(n-m)
\]
where  $m, n\in \Z^d$, ($|m|^2$ is shorthand for $m_1^2+ \cdots m_d^2$) and \[N(\lambda) = \#\{ m \in \Z^d: |m|^2=\lambda\}\]
is the number of lattice points on the sphere of radius $\lambda^{1/2}$ in $\R^{d}$, which is $\lambda^{d/2-1}$ by the Hardy-Littlewood asymptotic for all $d \geq 5$.  Moreover, by work of Kloosterman \cite{Kloos} this also holds for $d = 4$ as long as $\lambda\in \N\setminus 4\N$.

This can be thought of as a convolution operator with the measure
\[ \sigma_\lambda(n) := \frac{1}{N(\lambda)}{\bf 1}_{\{{\mathbf n} \in \Z^d: |n|^2=\lambda\}}(n). \]

We call a set of radii \emph{lacunary} if $\lambda_{j+1} > 2\lambda_j$ for all $j$ and define the discrete lacunary spherical maximal operator by \[
M_{lac}f = \sup_j |A_{\lambda_j}(f)| = \sup_j |f * \sigma_{\lambda_j}|.
\] 

This object was first studied in the continuous case when C. Calder\'on proved $L^p(\R^d)$ boundedness for all $p>1$ \cite{C}. 
The fine analysis of the $L^1$ endpoint is still an open question, see \cite{CK} for the current best result and some history on this subject.  In the discrete world, even less is known.  Originally, it was thought that this operator was also bounded on $l^p(\Z^d)$ for all $p >1$.  However, a counterexample of Zienkiewicz (see \cite{CH} for a description and an extension) shows that actually this operator is unbounded for all $1<p\leq \frac{d}{d-1}$, $d \geq 5$.  Therefore finding the exact range of boundedness became a more interesting question.  Kevin Hughes showed bounds for a restricted sequence of lacunary radii in \cite{Hughes_lacunary} for all $d \geq 4$.  Recently Kesler-Lacey-Mena showed $l^p$ bounds for any lacunary sequence for all $p> \frac{d-2}{d-3}$, $d \geq 5$ \cite{KLM} (see also the related work \cite{CH}).  Here we show $l^p$ bounds for all $p> \frac{d+1}{d-1}$.  Though this does not improve Kesler-Lacey-Mena's range for dimensions six and higher, our techniques give the first bounds for general lacunary sequences in $d=4$ and also provide $l^p(\Z^d)$ estimates for the error term of the multiplier decomposition in all dimensions. 

We will take advantage of a decomposition of the Fourier multiplier involving Kloosterman sums from analytic number theory.  Define $\Psi(\xi)$ be a smooth bump function supported on $\max_j |\xi_j| \leq 1/4$ and equal to 1 on $\max_j |\xi_j| \leq 1/8$.  Our multiplier has the following decomposition for $d \geq 5$ or $d = 4, \lambda \in \N\setminus 4\N$, essentially due to Magyar \cite{Magyar} (based on work of Magyar-Stein-Wainger \cite{MSW}).
\begin{theorem}
\label{MSW decomposition}
We have that 
\[\widehat{A_\lambda}(\xi) = \widehat{M_\lambda}(\xi)+\widehat{E_\lambda}(\xi) = \sum_{q = 1}^{\lambda^{1/2}}\sum_{l\in \Z^d_q}K(\lambda, q, l)\Psi(q\xi - l)\widehat{d\sigma}_\lambda(\xi - \frac{l}{q})+ \widehat{E_\lambda}(\xi)\]
where $K(\lambda, q ,l)$ is the exponential sum \[K(\lambda, q, l) = \frac{1}{q^d}\sum_{a \in U_q}\sum_{x\in \Z^d_q}e(\frac{-\lambda a+a|x|^2+lx}{q}),\] $\widehat{d\sigma}_\lambda$ is the Fourier transform of the continuous surface measure on the sphere of radius $\lambda^{1/2}$ and the error terms satisfy the decay property
\begin{equation}
    \label{error decay lacunary}
    \|E_\lambda f\|_{l^2(\Z^d)} \lesssim \lambda^{\frac{3-d}{4}+\varepsilon}\|f\|_{l^2(\Z^d)}.
\end{equation}
\end{theorem}

The main idea from this theorem is that our multiplier splits into a main term involving both arithmetic behaviour (from Kloosterman sums) and analytic behaviour (continuous Fourier transform of spherical measure) and an error term, arising from the circle method decomposition.  For more information see \cite{Magyar}.  Note that we abuse notation by relabeling our multiplier $\sigma = A$.  This is for notational flexibility as well as to avoid confusion in using both $\sigma$ to represent the continuous and discrete spherical surface measures.  We also emphasize the different normalization used here and in \cite{KLM}.

We will prove the following main theorem:
\begin{theorem}\label{main theorem}
Let $d \geq 5$ or $d=4$ and $\lambda \in \N\setminus 4\N$.  Then $M_{lac}:l^p(\Z^d)\to l^p(\Z^d)$ for all $p > \frac{d+1}{d-1}$.
\end{theorem}

Along the way we can also establish bounds for the error term in the approximation formula, which takes less significantly less work -- the argument only uses an $l^2$ bound coming from the Kloosterman refinement and the trivial $l^1$ bound to interpolate.  This result was first shown in \cite{Hughes_lacunary} (and extended to even show more there). 

\begin{corollary}
\label{main corollary}
The error term $E_{\lambda}$ in the Decomposition Theorem \ref{MSW decomposition}, treated as a convolution operator, is bounded in $l^{p}(\Z^{d})$ for all $p>\frac{d+1}{d-1}$.
\end{corollary}

We remark that the outline of our argument is also similar to the ones used in $l^p$-improving, such as \cite{Hughes} and \cite{KL}.  Finally, we comment briefly about the difficulty of extending this work to degree $k$ in the next section; for related work that discusses the difficulty of these improvements in relationship to unsolved problems in analytic number theory, see \cite{TCA}.   

\subsection{Acknowledgements}
Thank you to the anonymous referee(s) for pointing out an error from an incorrect normalization in the original manuscript, and also thank you to Kevin Hughes for helpful comments. The proof is now restructured here, but the main result is the same.

\section{Proof of Theorem \ref{main theorem}}

We start by explaining several facts that we will use later on that rely on the Kloosterman refinement.

\begin{lemma}
\label{Kloosterman degree 2}
    We have that
    \begin{equation}
    \sum_{q \leq \lambda^{1/2}}|K(\lambda, q, l)| \lesssim_{d,\varepsilon} \lambda^{\frac{3-d}{4}+\varepsilon}.
\end{equation}
\end{lemma}
\begin{proof}
Define $\rho(q,\lambda) = (q_1,\lambda)2^r$ where $q = q_12^r$ with $q_1$ odd.  From equation (1.13) in \cite{Magyar} we have that $\sup_{l}|K(\lambda,q,l)|\lesssim_{\epsilon} q^{-\frac{d}{2}+\frac{1}{2}+\epsilon}\rho(q, \lambda)^{1/2}.$ (see also \cite{KL}).  We also have the estimate 
\[
    \sum_{q \leq \lambda^{1/2}}q^{\beta}\rho(q,\lambda)^{1/2} \lesssim_{\beta,\varepsilon} \lambda^{\frac{\beta+1}{2}+\varepsilon}
\]
for any $\beta \in \R$ and all $\varepsilon > 0$ from equation (1.14) in \cite{Magyar}.  Combining these two with $\beta = \frac{1-d}{2}+\varepsilon$ we get the desired estimate.    
\end{proof}

We briefly explain how to get \eqref{error decay lacunary} from the estimate (1.9) in \cite{Magyar}, which asserts that (with our normalization) $\|\widehat{E_\lambda}\|_{L^{\infty}(\T^d)} \lesssim \lambda^{\frac{3-d}{4}+\varepsilon}$.  We have 
\begin{equation*}
\|E_\lambda f\|_{l^2(\Z^d)}
 = \|\widehat{E_\lambda f}\|_{L^2(\T^d)} = \|\widehat{E_\lambda}\widehat{f}\|_{L^2(\T^d)}
\leq \|\widehat{E_\lambda}\|_{L^\infty(\T^d)}\|f\|_{l^2(\Z^d)}
\end{equation*}
\begin{equation}
\label{error bound}
\lesssim \lambda^{\frac{3-d}{4}+\varepsilon}\|f\|_{l^2(\Z^d)}
\end{equation}
and the result follows from this. 
See also Lemma 2.9 in \cite{KL}. In fact using \eqref{error bound} we have
\begin{align*}
  \|\sup_{\Lambda \leq \lambda_j < 2\Lambda}E_\lambda f\|_{l^2(\Z^d)} &\leq \sum_{\Lambda\leq \lambda_j\leq 2\Lambda } \| E_{\lambda_j}f\|_{l^2{\Z^d}}\\
  &\leq \sum_{\Lambda\leq \lambda_j\leq 2\Lambda } \lambda^{\frac{3-d}{4}+\epsilon}_j\| f\|_{l^2{\Z^d}}\\
  &\leq \Lambda^{\frac{3-d}{4}+2\epsilon}\|f\|_{l^{2}(\Z^d)}\sum_{j}\lambda^{-\epsilon}_{j}
  &\lesssim_{\epsilon} \Lambda^{\frac{3-d}{4}+2\epsilon}\|f\|_{l^{2}(\Z^d)}.
\end{align*}
Note that we now have enough information to show Corollary \ref{main corollary}: let $N^2 = \Lambda$ and simply interpolate the above estimate with the trivial $l^1$ bound to get
\[
  \|\sup_{\Lambda \leq \lambda < 2\Lambda}E_\lambda f\|_{l^p(\Z^d)}\lesssim N^{(\frac{3-d}{2}+\varepsilon)(2-2/p)+2(2/p-1)}\|f\|_{l^p(\Z^d)}
\]
which leads to the bound for $\sup_j|E_{\lambda_j}|$ for all $p > \frac{d+1}{d-1}$.
\begin{remark}
Note that this error term comes from the Kloosterman refinement method, which allows one to simultaneously get a better error term that in \cite{MSW} and also to include the case $d=4$ (see also \cite{Hughes_lacunary}).

It is important to highlight the fact that the sums $K(\lambda, q ,l)$ are Kloosterman sums in degree 2, and although by a change of variable we can still get Kloosterman-like behavior in higher degrees, taking advantage of this seems to be a very difficult question in analytic number theory.  Specifically, in degree 2 one can complete the square to take advantage of the extra oscillation present in the character $e(a\lambda/q)$, which does not work in higher degrees.  Hence we focus on the degree 2 case.
\end{remark}

We also have the standard stationary phase estimate for the Fourier transform of the continuous surface measure:
\begin{equation}
\label{lemma: stationary phase estimate}
    |\widehat{d\sigma}_\lambda(\xi)|\lesssim |\lambda^{1/2}\xi|^{-(\frac{d-1}{2})}.
\end{equation}

Recalling the decomposition of Magyar from the Introduction, we label for future use
\begin{equation}
\label{main term}
\widehat{M_\lambda}(\xi) = \sum_{q \leq \lambda^{1/2}}\sum_{l\in \Z_q^d}K(q,\lambda,l)\Psi(q\xi - l)\widehat{d\sigma}_\lambda(\xi - \frac{l}{q}) = \sum_{q \leq \lambda^{1/2}}M_\lambda^q
\end{equation}

Following \cite{KLM} we will show the following two estimates for any $f = \chi_F$ and for natural number $\alpha$:
\begin{equation}
\label{l1 bound}
    \|M_1\|_{l^{1+\varepsilon}} \leq \alpha^2 \|f\|_{l^{1+\varepsilon}}
\end{equation}
and
\begin{equation}
\label{l2 bound}
    \|M_2\|_{l^{2}} \leq \alpha^{\varepsilon -\frac{d}{2}+\frac{3}{2}} \|f\|_{l^{2}}
\end{equation}

for operators $M_1$ and $M_2$ (to be defined later) such that $A_\tau f \leq M_1f+M_2f$, where $\tau: \Z^d \to \lambda_j$ ($\lambda_j$ being our lacunary sequence) is a stopping time.  Optimizing over $\alpha$ shows Theorem \ref{main theorem}; this is done at the end of our paper, also see \cite{KLM} for details.
We now decompose $A_\tau$ into pieces that will either become part of $M_1$ or $M_2$, showing they satisfy the appropriate bounds \eqref{l1 bound} or \eqref{l2 bound}.

Our first contribution to $M_1$ is $M_{1,1} = \chi_{\lambda^{1/2} \leq \alpha}A_\tau f$.  This satisfies \eqref{l1 bound} easily.  Next we look at the error term: our first contribution to $M_2$ is $M_{2,1} = |E_\tau f|$, which satisfies \eqref{l2 bound} due to \eqref{error bound}

Now we turn to the main term.  This will contribute to both $M_1$ and $M_2$.  First we have that
\begin{equation}
\label{main term large q}
    \|\sum_{\alpha<q \leq \lambda^{1/2}}M_\lambda^q\|_{l^2 \to l^2} \lesssim \alpha^{\varepsilon - \frac{d-3}{2}}\lambda^\varepsilon
\end{equation}
in a similar way as the proof of \eqref{Kloosterman degree 2} (see also Section 4.1 in \cite{Hughes}).  Hence our second contribution to $M_2$ is 
\[
M_{2,2} = \sum_{\alpha < q}|M_\tau^q|
\]
which clearly satisfies \eqref{l2 bound}.
All that is left is
\[
\sum_{1 \leq q \leq \alpha}M_\lambda^q f
\]
which will contribute to both $M_1$ and $M_2$.  Decompose $M_\tau ^q = M_{\tau,1}^1+ M_{\tau, 2}^q$ such that 
\[
\hat{M_{\lambda,1}^q}(\xi) = \sum_{l\in \Z_q^d}K(q,\lambda,l)\Psi_{\lambda^{1/2}/\alpha}(\xi - \frac{l}{q})\widehat{d\sigma}_\lambda(\xi - \frac{l}{q})
\]
where $\Psi_B(\xi) = \Psi(B\xi)$, is the low frequency piece and $M_{\lambda,2}^q = M_\lambda^q - M_{\lambda,1}^q$ is the high frequency piece.
The final contribution to $M_2$ is 
\[
M_{2,3} = |\sum_{1 \leq q \leq \alpha}M_{\tau,2}^q|.
\]
To show \eqref{l2 bound} we first note that by construction and the definition of $\Psi$,  $\Psi_q(\xi) - \Psi_{q\lambda^{1/2}/\alpha}(\xi) = 0$ if $|\xi| < \alpha/8q\lambda^{1/2}$.  Indeed, we have
\[
\|M_{2,3}\|_{l^2}^2 =  \sum_{\lambda^{1/2}_j>\alpha}\|\sum_{1 \leq q \leq \alpha}M_{\tau,2}^q\|_{l^2}^2
\]
\[
\leq \alpha\sum_{\lambda^{1/2}_j >\alpha}\int_{\T}\sum_{1 \leq q \leq \alpha}|\sum_{l\in \Z_q^d}K(q,\lambda_j,l)[\Psi_q(\xi-\frac{l}{q})-\Psi_{q\lambda_j^{1/2}/\alpha}(\xi - \frac{l}{q})]\widehat{d\sigma}_{\lambda_j}(\xi - \frac{l}{q})|^{2}
\]
by Cauchy-Schwartz and the trivial bound on the second factor.  Recalling the Kloosterman bound from the proof of Lemma \ref{Kloosterman degree 2}, the lower bound $|\xi-\frac{l}{q}| \geq \frac{\alpha}{8q\lambda^{1/2}}$, and the stationary phase estimate \eqref{lemma: stationary phase estimate} we get
\begin{align*}
&\lesssim\alpha\sum_{1\leq q \leq \alpha}\sum_{\lambda^{1/2}_{j}>\alpha}(q^{-\frac{d}{2}+\frac{1}{2}+\epsilon}\rho(q, \lambda_j)^{1/2}\alpha^{-\frac{d-1}{2}}q^{\frac{d-1}{2}})^2\\
&\leq\alpha\sum_{1\leq q \leq \alpha}\sum_{\lambda^{1/2}_{j}>\alpha}(q^{-\frac{d}{2}+\frac{1}{2}+\epsilon}\rho(q, \lambda_j)^{1/2}\alpha^{-\frac{d-1}{4}-10\epsilon}q^{\frac{d-1}{4}-10\epsilon})^2\\
&= \alpha^{\frac{3-d}{2}-20\epsilon}\sum_{1 \leq q <\alpha}\sum_{\lambda^{1/2}_j>\alpha}q^{-\frac{d}{2}+\frac{1}{2}-18\epsilon}\rho(q, \lambda_j)\\
&\leq \alpha^{\frac{3-d}{2}-20\epsilon}\sum_{\lambda^{1/2}_j>\alpha}\sum_{1 \leq q <\lambda^{1/2}_{j}}q^{-\frac{d}{2}+1-18\epsilon}\rho(q, \lambda_j)^{1/2}\\
&\leq \alpha^{\frac{3-d}{2}-20\epsilon}\sum_{\lambda_{j}>\alpha^2} \lambda_{j}^{\frac{-\frac{d}{2}+2-18\epsilon}{2}+\epsilon}\\
&\lesssim_{\epsilon} \alpha^{\frac{3-d}{2}-20\epsilon}.
\end{align*}
which satisfies \eqref{l2 bound}.  Note that we have used the fact that at most one term in the sum in $l$ is nonzero for each fixed $\xi$ and the bound in the proof of the Kloosterman estimate Lemma \ref{Kloosterman degree 2}.  Here it is important that we are dealing with a lacunary sequence.

The final piece is the last contribution to $M_1$, that is the low frequency piece
\[
M_{1,2} = |\sum_{1 \leq q \leq \alpha}M_{\tau,1}^q|.
\]
We will now comment on how to show \eqref{l1 bound} for $M_{1,2}$.  One can simply follow the argument detailed in \cite{KLM}.  Indeed, upon examination, this argument only relies on a specific kernel decomposition (equation 3.11), a bound on Ramanujan sums (Lemma 3.13), and a certain $l^p$ bound stemming from the kernel decomposition (Proposition 3.15).  Equation 3.11 is detailed in \cite{Hughes_lacunary} and is valid in dimension 4, Lemma 3.13 is independent of dimension, and Proposition 3.15 only relies on the Hardy-Littlewood asymptotic, which is valid in dimension 4 for the radii that we consider, that is $\lambda \in \N\setminus 4\N$.

Let $f=1_{F}$, using  \eqref{l1 bound} and \eqref{l2 bound} we obtain that 
\begin{equation*}
    \sup_{\beta>0}\beta|\{x; M_{lac}f(x)>\beta\}|^{\frac{d-1}{d+1}}\leq 2\|f\|_{\frac{d+1}{d-1}}.
\end{equation*}
Indeed, fix $\beta>0$, we can choose $\alpha=\beta^{-\frac{1}{d-1}}$ to obtain (up to an $\varepsilon$ loss)
\begin{align*}
    \beta|\{x; M_{lac}f(x)>\beta\}|^{\frac{d-1}{d+1}}&\leq  \beta|\{x; M_1f(x)>\beta\}|^{\frac{d-1}{d+1}}+
    \beta|\{x; M_2f(x)>\beta\}|^{\frac{d-1}{d+1}}\\
    &\leq \beta\left(\frac{\|M_{1}f\|_1}{\beta}\right)^{\frac{d-1}{d+1}}+\beta\left(\frac{\|M_2f\|_2}{\beta}\right)^{2\left(\frac{d-1}{d+1}\right)}\\
    &\leq \beta^{\frac{2}{d+1}}(\alpha^{2}|F|)^{\frac{d-1}{d+1}}+\beta^{-\frac{(d-3)}{d+1}}\alpha^{-\frac{(d-3)(d-1)}{d+1}}|F|^{\frac{d-1}{d+1}}\\
    &=2|F|^{\frac{d-1}{d+1}}\\
    &=2\|f\|_{\frac{d+1}{d-1}}.
\end{align*}
and this gives Theorem \ref{main theorem} for all $p> \frac{d+1}{d-1}$.

\bibliographystyle{amsplain}


\end{document}